\documentclass[12pt]{amsart}
\usepackage{epstopdf}
\usepackage[numbers,sort&compress]{natbib}
\bibpunct[, ]{[}{]}{,}{n}{,}{,}
\makeatletter
\def\NAT@def@citea{\def\@citea{\NAT@separator}}
\makeatother
\usepackage{xcolor}

\usepackage{amsmath}
\usepackage{amssymb}
\usepackage{amsthm}
\usepackage{amsfonts}
\usepackage{arydshln}
\usepackage{enumerate}

\usepackage{hyperref}
\usepackage[thinlines]{easybmat}
\numberwithin{equation}{section}
\newcommand{\hsp}{\mathcal{H}}
\newcommand{\oL }{\mathcal{L}(\hsp) }

\newcommand{\dfc}[1]{D_{\Theta_#1}}

\newcommand{\dcf}[1]{\widetilde{D}_{\Theta_#1}}
\newcommand{\dfct}[1]{{D}_{\widetilde{\Theta}_#1}}

\newcommand{\h}{\mathcal{H}}

\newcommand{\ke}{{\mathbf{k}}}
\newcommand{\C}{{\mathbf{C}}}
\newcommand{\F}{{\mathbf{F}}}
\newcommand{\J}{{\mathbf{J}}}

\newcommand{\G}{{\mathbf{G}}}

\newcommand{\z}{{z}}
\newtheorem{theorem}{Theorem}[section]
\newtheorem{proposition}[theorem]{Proposition}
\newtheorem{corollary}[theorem]{Corollary}
\newtheorem{lemma}[theorem]{Lemma}
\theoremstyle{definition}

\newtheorem{remark}[theorem]{Remark}
\numberwithin{equation}{section}

\title[Characterizations of MATHO's]{Characterizations of Matrix valued asymmetric truncated Hankel operators}

\author[R. Khan]{Rewayat Khan}
\address{R. Khan, Abbottabad University of Science and Technology, Pakistan}
\email{rewayat.khan@gmail.com}

\author[J. E. Lee]{Ji Eun Lee}
\address{Ji Eun Lee,  Department of Mathematics and Statistics, Sejong University, Seoul, 05006,  Republic of Korea }
\email{jieunlee7@sejong.ac.kr}

\keywords{model spaces, conjugations, matrix valued asymmetric truncated  Toeplitz operators, matrix valued asymmetric truncated Hankel operators.}
\subjclass[2010]{Primary 47B35, Secondary 47B32, 30D20}
\begin{document}
\begin{abstract}{In this paper we introduce the class of matrix valued asymmetric truncated Hankel operators. By using characterizations of matrix valued asymmetric truncated Toeplitz operators, we characterize matrix valued asymmetric truncated Hankel operators in the case
when two involved inner matrices are $J$-symmetric.}
\end{abstract}
\maketitle

\section{Introduction}

  As usual, $\mathbb{D}=\{\lambda\in\mathbb{C}:|\lambda|<1\}$ and $\mathbb{T}=\partial\mathbb{D}=\{\z\in\mathbb{C}: |\z|=1\}$ denote the unit disc and unit circle, respectively. { Let} $H^2$ denote the classical Hardy space in the unit disk $\mathbb{D}$.  Truncated Hankel operators (THO's) (respectively truncated Toeplitz operators (TTO's)) and asymmetric truncated Hankel operators (ATHO's)(respectively asymmetric truncated Toeplitz operators (ATTO's)) are compressions of multiplication operator to the backward shift invariant subspaces of $H^2$ (with two possibly different underlying subspaces called model spaces in the asymmetric case). Each of these subspaces is of the form $K_{\theta}=(\theta H^{2})^{\perp}=H^2\ominus \theta H^2$, where $\theta$ a complex-valued inner function: $\theta\in H^\infty$ and $|\theta(\z)|=1$ a.e. on the unit circle $\mathbb{T}$. Since D. Sarasons paper \cite{sar} TTO's, and later on ATTO's \cite{part,part2, BCKP}, MTTO's \cite{KT} and MATTO's \cite{RK2, RK3}, have been intensly studied.

The theory of THO's and ATHO's in the scalar case can be extended to vector version, defined on subspaces of a vector valued Hardy space $H^2(\hsp)$ with a finite dimensional complex Hilbert space $\hsp$. These subspaces are corresponding to \textit{inner functions}: An analytic function $\Theta$ having values in $\oL$ (the algebra of all bounded linear operators on $\hsp$), bounded and such that the boundary values $\Theta(z)$ are unitary operators a.e. on $\mathbb{T}$. We will consider only \textit{pure inner functions}, that is, $\Theta$ such that $\|\Theta(0)\|_{\oL}< 1$.
 The model space
$$K_{\Theta}=[ \Theta H^{2}(\hsp)]^{\perp}$$
corresponding to an inner function $\Theta$ is invariant under the backward shift $S^*$. Moreover, by the vector valued version of Beurling's invariant subspace theorem, each closed (nontrivial) $S^*$--invariant subspace of $H^2(\hsp)$ is a model space (\cite[Chapter 5, Theorem 1.10]{berc}). Let $P_{\Theta}$ be the orthogonal projection from $L^{2}(\oL)$ onto $K_{\Theta}$.

 Sections 2 and 3 contain preliminary material on spaces of vector valued functions (Section 2), model spaces and MATTO's (Section 3). Section 4 contains some results related to matrix valued asymmetric truncated Hankel operators. In section 5, we just remind the characterization of matrix valued asymmetric truncated Toeplitz operators (see \cite{RK3}). In section 6, connection between matrix valued asymmetric truncated Toeplitz (MATTO's)  and matrix valued asymmetric truncated Hankel operators (MATHO's) is presented. Section 7 and section 8 is devoted to  characterizes matrix valued asymmetric truncated Hankel operators by using compressed shift and modified compressed shift respectively.

\section{Spaces of vector valued functions and their operators}\label{s2}

Let $\hsp$ be a complex separable Hilbert space. In what follows $\|\cdot\|_\hsp$ and $\langle\cdot,\cdot\rangle_{\hsp}$ will denote the norm and the inner product in $\hsp$, respectively. Moreover, we will assume that $\dim\hsp<\infty$. The space $L^{2}(\hsp)$ can be defined as
$$L^2(\h)=\{f\colon\mathbb{T}\to \hsp: f \text{ is measurable and} \int_{\mathbb{T}} \| f(\z)\|_\hsp^2\,dm(\z)<\infty  \}$$
($m$ being the normalized Lebesgue measure on $\mathbb{T}$). As usual, each $f\in L^2(\hsp)$ is interpreted as a class of functions equal to the representing $f$ a.e. on $\mathbb{T}$ with respect to $m$. The space $L^2(\h)$ is a (separable) Hilbert space with the inner product given by
\begin{equation*}
\langle f,g\rangle_{L^{2}(\hsp)}=\int_{\mathbb{T}}\langle f(\z),g(\z)\rangle_{\hsp}\,dm(\z),\quad f,g\in L^2(\hsp)
\end{equation*}
and the induced norm is given by
$$\|f\|_{L^2(\hsp)}^2=\int_{\mathbb{T}} \| f(\z)\|_\hsp^2\,dm(\z)=\sum\limits_{n=-\infty}^{\infty}\|a_{n}\|_\hsp^{2}.$$

The vector valued Hardy space $H^2(\hsp)$ is defined as the set of all the elements of $L^2(\hsp)$ whose Fourier coefficients with negative indices vanish. Each $f\in H^2(\hsp)$, $\displaystyle f(\z)=\sum_{n=0}^{\infty} a_n \z^n$, can also be identified with a function
$$f(\lambda)=\sum\limits_{n=0}^\infty a_n\lambda^n,\quad \lambda\in\mathbb{D},$$
analytic in the unit disk $\mathbb{D}$.

 We can also consider the spaces of all essentially bounded functions on the $\mathbb{T}$. This space is denoted by
 $L^{\infty}(\hsp)$ and
 $$H^{\infty}(\hsp)=L^{\infty}(\hsp)\cap H^{2}(\hsp).$$

 Here we assume that $\dim\hsp<\infty$ so we can consider $\oL $ as a Hilbert space with the Hilbert--Schmidt norm and we may also define as above the spaces $L^2(\oL)$ and $H^2(\oL)$:
$$H^2(\oL)=\left\{\F\in L^2(\oL): \F(\z)=\sum_{n=0}^{\infty} F_n \z^n\right\}.$$
We thus have
$$L^2(\oL)=[zH^2(\oL)]^{*}\oplus H^{2}(\oL).$$

For each $\lambda\in \mathbb{D}$ the function $\ke_\lambda(z)= (1-\bar{\lambda}z)^{-1}I_{\hsp}$ is bounded and belongs to $H^{2}(\oL)$, and has the following reproducing property
$$\langle f, \ke_\lambda x\rangle_{L^2(\hsp)}=\langle f(\lambda),x\rangle_{\hsp}, \quad f\in H^2(\hsp).$$

To each $\Phi\in L^\infty(\oL)$ there corresponds a multiplication operator $M_\Phi:L^2(\hsp)\to L^2(\hsp)$: for $f\in L^2(\hsp)$,
$$(M_\Phi f)(\z)=\Phi(\z)f(\z)\quad \text{a.e. on }\mathbb{T}.$$
Let $P_{+}$ denote the orthogonal projection from $L^2(\hsp)$ to $H^{2}(\hsp)$. Then
the Toeplitz operator $T_{\Phi}$ is the compression of $M_{\Phi}$ to the Hardy space: $T_{\Phi}:H^2(\hsp)\to H^2(\hsp)$, $$T_{\Phi}f=P_+M_{\Phi}f\quad\text{for }f\in H^2(\hsp).$$

It is clear that $(M_\Phi)^*=M_{\Phi^*}$ and $(T_\Phi)^*=T_{\Phi^*}$, where $\Phi^*(\z)=\Phi(\z)^*$ a.e. on $\mathbb{T}$. The operators $S$ (forward shift) and $S^{*}$ (backward shift) are examples of Toeplitz operators with symbols $zI_{\hsp}$ and $\overline{z}I_{\hsp}$, respectively, where $I_{\hsp}$ is the identity operator on $\hsp$.

\medskip

\section{ MATTO's and MATHO's}\label{s3}

For two operator valued inner functions $\Theta_1,\Theta_2\in H^{\infty}(\oL)$ and $\Phi\in L^2(\oL )$ let
\begin{equation}\label{11}
A_{\Phi}^{\Theta_{1}, \Theta_{2}}f=P_{\Theta_{2}}(\Phi f),\quad f\in K_{\Theta_{1}}\cap H^{\infty}(\oL ).
\end{equation}

 The operator $A_{\Phi}^{\Theta_{1}, \Theta_{2}}$ is called a matrix valued asymmetric truncated Toeplitz operator (MATTO). When $\Theta_1=\Theta_2=\Theta$, then $A_{\Phi}^{\Theta}=A_{\Phi}^{\Theta, \Theta}$ is called a matrix valued truncated Toeplitz operator (MTTO's) \cite{KT}. Both are densely defined. Let $\mathcal{MT}(\Theta_{1},\Theta_{2})$ be the set of all MATTO's of the form \eqref{11} which can be extended boundedly to the whole space $K_{\Theta_{1}}$ and for $\Theta_{1}=\Theta_{2}=\Theta$ let $\mathcal{MT}(\Theta )=\mathcal{MT}(\Theta,\Theta)$.

 Two important examples of operators from $\mathcal{MT}(\Theta )$ are the model operators
\begin{equation}\label{modl}
S_{\Theta}=A_{z }^{\Theta }=A_{zI_{\hsp}}^{\Theta }\quad\text{and}\quad S_{\Theta }^*=A_{\bar z }^{\Theta }=A_{\bar zI_{\hsp}}^{\Theta  }.
\end{equation}
These operators are models for a class of Hilbert space contractions. For example, each $C_0$ contraction with finite defect indices is unitarily equivalent to $S_{\Theta}$ for some
 operator valued inner function $\Theta$ (see \cite[Chapter IV]{NF}).

Truncated Hankel operators were first defined and studied by C. Gu in \cite{Gu}. Their asymmetric version were introduced in \cite{zero}.

An asymmetric matrix valued truncated Hankel operator $B_{\Phi}^{\Theta_{1}, \Theta_{2}}$ with symbol $\Phi\in L^2(\oL)$ is an operator from $K_{\Theta_1}$ into $K_{\Theta_2}$ densely defined by
\begin{equation}\label{MTHO: defin}
B_{\Phi}^{\Theta_{1}, \Theta_{2}}f=P_{\Theta_2}\J(I-P_{+})(\Phi f), \quad f\in K_{\Theta_1}^{\infty},
\end{equation}
where $\J:L^2(E)\to L^2(E)$ is the flip operator given by
$$\J f(z)=\overline{z}f(\overline{z}), \quad |z|=1.$$
Let $\mathcal{MH}(\Theta_1, \Theta_2)$ be the set of all MATHO's of the form (\ref{MTHO: defin}) which can be extended boundedly to the whole space $K_{\Theta_1}$ and $\mathcal{MH}(\Theta)$ denote the set of all MTHO's.
Some algebraic properties of THO's were studied in \cite{Gu}. Here we continue the investigation started in \cite{BM, zero}.

For each $\lambda\in\mathbb{D}$ we can consider
$$\ke_{\lambda}^{\Theta}(z)=\tfrac1{1-\bar\lambda z}(I_\hsp-\Theta(z)\Theta(\lambda)^*)\in H^{\infty}(\oL).$$
Then, for each $x\in\hsp$ it has the following reproducing property
$$\langle f, \ke_{\lambda}^{\Theta}x\rangle_{L^2(\hsp)}=\langle f(\lambda), x\rangle_\hsp\quad\text{for every }f\in K_{\Theta}.$$
It follows in particular that $K_{\Theta}^{\infty}=K_{\Theta}\cap H^{\infty}(\hsp)$ is a dense subset of $K_{\Theta}$.

Let $$\mathcal{D}_{\Theta}=\{(I_{\hsp}-\Theta\Theta(0)^*)x:\ x\in\hsp\}=\{\ke_{0}^{\Theta}x:\ x\in\hsp\}\subset K_{\Theta}.$$
For each $f\in K_{\Theta}$ we have $f\perp \mathcal{D}_{\Theta}$ if and only if $f(0)=0$.
It follows that
$$(S_{\Theta}^*f)(z)=\left\{\begin{array}{cl}
\bar z f(z)&\text{for }f\perp\mathcal{D}_{\Theta},\\
-\bar z\big(\Theta(z)-\Theta(0)\big)\Theta(0)^*x&\text{for }f=\ke_{0}^{\Theta}x\in\mathcal{D}_{\Theta}.
\end{array}\right.$$
The defect operator of $S_{\Theta}$ is denoted by  $$D_{\Theta}:=I_{K_{\Theta}}-S_{\Theta}S_{\Theta}^*.$$
Since for each $f\in H^2(\hsp)$ we have $(I_{H^2(\hsp)}-SS^*)f=f(0)$ (a constant function in $H^2(\hsp)$), it follows that for $f\in K_{\Theta}$,
\begin{equation}\label{ect}
\begin{split}
D_{\Theta}f&=(I_{K_{\Theta}}-S_{\Theta}S_{\Theta}^*)f=P_{\Theta}(I_{H^2(\hsp)}-SS^*)f\\
&=(I_{\hsp}-\Theta\Theta(0)^*)f(0)=\ke_{0}^{\Theta}f(0)\in \mathcal{D}_{\Theta}.
\end{split}
\end{equation}
More precisely,
$$D_{\Theta}f=\left\{\begin{array}{cl}
0&\text{for }f\perp\mathcal{D}_{\Theta},\\
\ke_{0}^{\Theta}(I_{\hsp}-\Theta(0)\Theta(0)^*)x&\text{for }f=\ke_{0}^{\Theta}x\in\mathcal{D}_{\Theta}.
\end{array}\right.$$

\medskip

\section{MATHO's and Model spaces}\label{s4}

In \cite{RK} the author considers the generalized Crofoot transform. For a pure inner function $\Theta\in H^{\infty}(\oL )$ and $W\in \oL $ such that $\|W\|_{\oL}<1$ define the generalized Crofoot transform $J_{W}^{\Theta}: L^2(\hsp)\to L^2(\hsp )$ by
\begin{equation}\label{J_W}
J_{W}^{\Theta}f=D_{W^{*}}(I_{L^2(\hsp)}-\Theta W^{*})^{-1}f,\quad f\in L^2(\hsp).
\end{equation}
The adjoint of $J_{W}^{\Theta}$ is given by the formula
\begin{equation}\label{J_W adj}
(J_{W}^{\Theta})^{\ast}g=D_{W^{*}}(I_{L^2(\hsp)}+\Theta^{W} W^{*})^{-1}g,\quad g\in L^2(\hsp).
\end{equation}
Then $J_{W}^{\Theta}$ is unitary and maps $K_{\Theta}$ onto $K_{\Theta^{W}}$ where
\begin{equation}\label{theta-W}
\Theta^{W}(z)=-W+D_{W^{*}}(I_{L^2(\hsp)}-\Theta(z) W^{*})^{-1}\Theta(z) D_{W}.
\end{equation}
\smallskip

The following theorem describes the action of the Crofoot transform on $\mathcal{MH}(\Theta_{1}, \Theta_{2})$ (see \cite{RK} for the case $\Theta_1=\Theta_2$ and \cite{blicharz1} for the scalar case):
\begin{theorem}
	Let $\Theta_{1},\Theta_{2}\in H^{\infty}(\oL)$ be two pure inner functions and let $W_1,W_2\in \oL$ be such that $\|W_1\|_{\oL}<1$ and $\|W_2\|_{\oL}<1$. A bounded linear operator $B:K_{\Theta_1}\to K_{\Theta_2}$ belongs to $\mathcal{MH}(\Theta_{1}, \Theta_{2})$ if and only if $J_{W_{2}}^{\Theta_2}B(J_{W_{1}}^{\Theta_1})^{*}$ belongs to $\mathcal{MH}(\Theta_{1}^{W_1}, \Theta_{2}^{W_2})$. More precisely, $B=B_{\Phi}^{\Theta_{1}, \Theta_{2}}\in \mathcal{MH}(\Theta_{1}, \Theta_{2})$ if and only if $$J_{W_{2}}^{\Theta_2}B(J_{W_{1}}^{\Theta_1})^{*}=B_{\Psi}^{\Theta_{1}^{W_1}, \Theta_{2}^{W_2}}\in\mathcal{MH}(\Theta_{1}^{W_1}, \Theta_{2}^{W_2})$$ where $\Psi$ is given by
\begin{equation*}\label{Psi 0}
\Psi(z) = D_{W_{2}^{*}}(I_{\oL}-\Theta_{2}(z)W_{2}^{*})^{-1}\Phi(z) D_{W_{1}^{*}}(I_{\oL}+\Theta^{W_1}_{1}(z) W_{1}^{*})^{-1}.
\end{equation*}
\end{theorem}

\begin{proof}
Let $f\in K_{\Theta_{1}^{W_1}}$ and $x\in \hsp$. Then
\begin{eqnarray*}
&&\langle J_{W_2}^{\Theta_2}B_{\Phi}^{\Theta_{1}, \Theta_{2}}(J_{W_1}^{\Theta_1})^{\ast}f, k_{\lambda}^{\Theta^{W_2}_{2}}x\rangle\cr
&=&\langle B_{\Phi}^{\Theta_{1}, \Theta_{2}}(J_{W_1}^{\Theta_1})^{\ast}f, (J_{W_2}^{\Theta_2})^{\ast}k_{\lambda}^{\Theta^{W_2}_{2}}x\rangle\\
&=&\langle P_{\Theta_{2}}\J(I-P_{+})(\Phi (J_{W_1}^{\Theta_1})^{\ast}f), (J_{W_2}^{\Theta_2})^{\ast}k_{\lambda}^{\Theta^{W_2}_{2}}x\rangle\\
&=&\langle \J(\Phi (J_{W_1}^{\Theta_1})^{\ast}f), k_{\lambda}^{\Theta_{2}}(I-W_{2}\Theta_{2}^{*})^{-1}D_{W_{2}^{*}}x\rangle \\
&=&\int_{\mathbb{T}}\langle \bar z(\Phi (J_{W_1}^{\Theta_1})^{\ast}f)(\bar z), (J_{W_2}^{\Theta_2})^{\ast}k_{\lambda}^{\Theta^{W_2}_{2}}x\rangle \\
&=&\int_{\mathbb{T}}\langle \bar z\Phi(\bar z) ((J_{W_1}^{\Theta_1})^{\ast}f)(\bar z), (J_{W_2}^{\Theta_2})^{\ast}k_{\lambda}^{\Theta^{W_2}_{2}}x\rangle \\
&=&\int_{\mathbb{T}}\langle \bar z \Phi(\bar z) D_{W_{1}^{*}}(I+\Theta^{\prime}_{1}(\bar z)W_{1}^{*})^{-1}f(\bar z), (J_{W_2}^{\Theta_2})^{\ast}k_{\lambda}^{\Theta^{W_2}_{2}}x\rangle \\
&=&\int_{\mathbb{T}}\langle J_{W_2}^{\Theta_2}(\bar z \Phi(\bar z) D_{W_{1}^{*}}(I+\Theta^{\prime}_{1}(\bar z)W_{1}^{*})^{-1}f(\bar z)), k_{\lambda}^{\Theta^{W_2}_{2}}x\rangle \\
&=&\int_{\mathbb{T}}\langle D_{W_{2}^{*}}(I-\Theta_{2}(\bar z)W_{2}^{*})^{-1} \bar z\Phi(\bar z) D_{W_{1}^{*}}(I+\Theta^{\prime}_{1}(\bar z)W_{1}^{*})^{-1}f(\bar z), k_{\lambda}^{\Theta^{W_2}_{2}}x\rangle\\
&=&\int_{\mathbb{T}}\langle \bar z D_{W_{2}^{*}}(I-\Theta_{2}(\bar z)W_{2}^{*})^{-1} \Phi(\bar z) D_{W_{1}^{*}}(I+\Theta^{\prime}_{1}(\bar z)W_{1}^{*})^{-1}f(\bar z), k_{\lambda}^{\Theta^{W_2}_{2}}x\rangle\\
&=&\int_{\mathbb{T}}\langle \J (D_{W_{2}^{*}}(I-\Theta_{2}(z)W_{2}^{*})^{-1} \Phi(z) D_{W_{1}^{*}}(I+\Theta^{\prime}_{1}(z)W_{1}^{*})^{-1}f(z)), k_{\lambda}^{\Theta^{W_2}_{2}}x\rangle\\
&=&\int_{\mathbb{T}}\langle P_{\Theta^{W_2}_{2}}\J (\Psi(z)f(z)), k_{\lambda}^{\Theta^{W_2}_{2}}x\rangle\\
&=&\langle P_{\Theta^{W_2}_{2}}\J (\Psi f), k_{\lambda}^{\Theta^{W_2}_{2}}x\rangle\\
&=&\langle P_{\Theta^{W_2}_{2}}\J (I-P_{+})(\Psi f), k_{\lambda}^{\Theta^{W_2}_{2}}x\rangle\\
&=&\langle B^{\Theta^{W_1}_{1}, \Theta^{W_2}_{2}}_{\Psi}f, k_{\lambda}^{\Theta^{W_2}_{2}}x\rangle,
\end{eqnarray*}
where $\Psi$ is given by $$\Psi(z) = D_{W_{2}^{*}}(I_{\oL}-\Theta_{2}(z)W_{2}^{*})^{-1}\Phi(z) D_{W_{1}^{*}}(I_{\oL}+\Theta^{W_1}_{1}(z) W_{1}^{*})^{-1}.$$
Hence we complete the proof.
\end{proof}

\smallskip
Recall that if $\Theta\in H^{\infty}(\oL )$ is an inner function, then so is
$$\widetilde{\Theta}(z)=\Theta(\bar z)^*.$$
Let us consider the map $\tau_{\Theta}:L^2(\hsp)\to L^2(\hsp)$ defined for $f\in L^2(\hsp)$ by
\begin{equation}
\label{tau}(\tau_{\Theta}f)(z)=\bar z\Theta(\bar z)^*f(\bar z)=\bar z\widetilde{\Theta}( z)f(\bar z)\quad\text{a.e. on }\mathbb{T}.
\end{equation}
The map $\tau_{\Theta}$ is an isometry and its adjoint $\tau_{\Theta}^*=\tau_{\widetilde{\Theta}}$ is also its inverse. Hence $\tau_{\Theta}$ is unitary. Moreover, it is easy to verify that
$$\tau_{\Theta}(\Theta H^2(\hsp))\subset H^2(\hsp)^{\perp}\quad\text{and}\quad \tau_{\Theta}( H^2(\hsp)^{\perp})\subset \widetilde{\Theta} H^2(\hsp),$$
which implies that
$$ \tau_{\Theta}(K_{\Theta})=K_{\widetilde{\Theta}}.$$

\begin{theorem}\label{ttau}
	Let $\Theta_{1},\Theta_{2}\in H^{\infty}(\oL)$ be two pure inner functions. A bounded linear operator $B:K_{\Theta_1}\to K_{\Theta_2}$ belongs to $\mathcal{MH}(\Theta_{1}, \Theta_{2})$ if and only if $\tau_{\Theta_2}B\,\tau_{\Theta_1}^{*}$ belongs to $\mathcal{MH}(\widetilde{\Theta}_{1}, \widetilde{\Theta}_{2})$. More precisely, $B=B_{\Phi}^{\Theta_{1}, \Theta_{2}}\in \mathcal{MH}(\Theta_{1}, \Theta_{2})$ if and only if $\tau_{\Theta_2}B\,\tau_{\Theta_1}^{*}=B_{\Psi}^{\widetilde{\Theta}_{1}, \widetilde{\Theta}_{2}}\in\mathcal{MH}(\widetilde{\Theta}_{1}, \widetilde{\Theta}_{2})$ with
	\begin{equation}\label{symbda}
	\begin{split}
	\Psi(z) &=\Theta_2( z)^*\Phi(\bar z)\widetilde{\Theta}_1(z)^* \\&=\widetilde{\Theta}_2( \bar z)\Phi(\bar z)\Theta_1( \bar z) \quad\text{a.e. on }\mathbb{T}.
	\end{split}
	\end{equation}
\end{theorem}

\begin{proof}
	Let $B:K_{\Theta_1}\to K_{\Theta_2}$ be a bounded linear operator. Assume that $B=B_{\Phi}^{\Theta_{1}, \Theta_{2}}\in \mathcal{MT}(\Theta_{1}, \Theta_{2})$ with some $\Phi\in L^2(\oL)$, and take $f\in K_{\widetilde{\Theta}_1}^{\infty}$ and $g\in K_{\widetilde{\Theta}_2}^{\infty}$. Note that $\tau_{\widetilde{\Theta}_1}f\in K_{{\Theta}_1}^{\infty}$ and $\tau_{\widetilde{\Theta}_2}g\in K_{{\Theta}_2}^{\infty}$. Then
	\begin{align*}
	&\langle \tau_{\Theta_2}B\,\tau_{\Theta_1}^{*}f,g \rangle_{L^2(\hsp)}\\
&=\langle B_{\Phi}^{\Theta_{1}, \Theta_{2}}\tau_{\widetilde{\Theta}_1}f,\tau_{{\Theta}_2}^*g \rangle_{L^2(\hsp)}\\&=\langle P_{\Theta_2}\J(I-P_{+})({\Phi}\,\tau_{\widetilde{\Theta}_1}f),\tau_{{\Theta}_2}^*g \rangle_{L^2(\hsp)}\\
&=\langle P_{\Theta_2}\J({\Phi}\,\tau_{\widetilde{\Theta}_1}f),\tau_{{\Theta}_2}^*g \rangle_{L^2(\hsp)}\\
&=\langle \tau_{{\Theta}_2}\J({\Phi}\,\tau_{\widetilde{\Theta}_1}f),g \rangle_{L^2(\hsp)}\\
	&=\int_{\mathbb{T}}\langle \tau_{{\Theta}_2}( \bar z\Phi(\bar z)(\tau_{\widetilde{\Theta}_1}f)(\bar z)),g(z) \rangle_{\hsp}\,dm(z)\\
	&=\int_{\mathbb{T}}\langle \bar z \Theta_{2}(\bar z)^{*}( z\Phi(z)(\tau_{\widetilde{\Theta}_1}f)(z)),g(z) \rangle_{\hsp}\,dm(z)\\
&=\int_{\mathbb{T}}\langle \Theta_{2}(\bar z)^{*}\Phi(z)\widetilde{\Theta}_{1}(\bar z)^{*}\bar z f(\bar z),g(z) \rangle_{\hsp}\,dm(z)\\
&=\int_{\mathbb{T}}\langle \bar z \Theta_{2}(\bar z)^{*}\Phi(z)\widetilde{\Theta}_{1}(\bar z)^{*} f(\bar z),g(z) \rangle_{\hsp}\,dm(z)\\
&=\int_{\mathbb{T}}\langle \J (\Theta_{2}( z)^{*}\Phi(\bar z)\widetilde{\Theta}_{1} (z)^{*} f(z)),g(z) \rangle_{\hsp}\,dm(z)\\
&=\int_{\mathbb{T}}\langle P_{\Theta_2} \J (I-P_{+}) (\Theta_{2}( z)^{*}\Phi(\bar z)\widetilde{\Theta}_{1} (z)^{*} f(z)),g(z) \rangle_{\hsp}\,dm(z)\\
&=\int_{\mathbb{T}}\langle P_{\widetilde{\Theta}_2} \J (I-P_{+}) (\Psi(z) f(z)),g(z) \rangle_{\hsp}\,dm(z)\\
&=\langle B_{\Psi}^{\widetilde{\Theta}_1, \widetilde{\Theta}_2}f, g\rangle,
	\end{align*}
	where $\Psi\in L^2(\oL)$ is given by \eqref{symbda}.
	
	Now, if $\tau_{\Theta_2}B\,\tau_{\Theta_1}^{*}=B_{\Psi}^{\widetilde{\Theta}_{1}, \widetilde{\Theta}_{2}}\in\mathcal{MH}(\widetilde{\Theta}_{1}, \widetilde{\Theta}_{2})$ for some $\Psi\in L^2(\oL)$, then  $B=\tau_{\widetilde{\Theta}_2}B_{\Psi}^{\widetilde{\Theta}_{1}, \widetilde{\Theta}_{2}}\tau_{\widetilde{\Theta}_1}^{*}$  by the first part of the proof of $B=B_{\Phi}^{\Theta_{1}, \Theta_{2}}\in \mathcal{MH}(\Theta_{1}, \Theta_{2})$ with
	\begin{equation}\label{wykwyk}
	\begin{split}
	\Phi(z) =\widetilde{\Theta}_2(z)^*\Psi(\bar z)\Theta_1(z)^{*}\quad\text{a.e. on }\mathbb{T}.
	\end{split}
	\end{equation}
\end{proof}

\medskip

Denote
$$\widetilde{D}_{\Theta}= I - S_{\Theta}^*S_{\Theta}.$$
Applying Theorem \ref{ttau} to the model operator $S_{\Theta}$ ($\Theta_{1}=\Theta_{2}=\Theta$) we obtain that
\begin{equation}\label{sz}
\tau_{\Theta}S_{\Theta}\tau_{\Theta}^*=\tau_{\Theta}S_{\Theta}\tau_{\widetilde{\Theta}}=S_{\widetilde{\Theta}}^*
\end{equation}
(see \cite[p. 1001]{KT}). It follows that
\begin{equation}\label{ddd}
\widetilde{D}_{\Theta}=\tau_{\widetilde{\Theta}}{D}_{\widetilde{\Theta}}\tau_{\Theta}=\tau_{\widetilde{\Theta}}{D}_{\widetilde{\Theta}}\tau_{\widetilde{\Theta}}^*
\end{equation}
and by \eqref{ect},
$$\widetilde{D}_{\Theta}f=\tau_{\widetilde{\Theta}}\big(\ke_0^{\widetilde{\Theta}}(\tau_{\Theta}f)(0)\big)\quad\text{for all }f\in K_{\Theta}.$$
For $\lambda\in\mathbb{D}$, let $\widetilde{\ke}_{\lambda}^{\Theta}(z)x=\tau_{\widetilde{\Theta}}(\ke_{\bar\lambda}^{\widetilde{\Theta}}(z)x)$ for $x\in\hsp$. Then (a.e. on $\mathbb{T}$)
\begin{align*}
\widetilde{\ke}_{\lambda}^{\Theta}(z)x&=\tau_{\widetilde{\Theta}}(\ke_{\bar\lambda}^{\widetilde{\Theta}}(z)x)=\bar z \Theta(z)\ke_{\bar\lambda}^{\widetilde{\Theta}}(\bar z)x\\
&=\tfrac{\bar z}{1-\lambda\bar z}\Theta(z)(I_{\hsp}-\widetilde{\Theta}(\bar z)\widetilde{\Theta}(\bar \lambda)^*)x\\
&=\tfrac{1}{z-\lambda}\Theta(z)(I_{\hsp}-{\Theta}( z)^*{\Theta}( \lambda))x
\\
&=\tfrac{1}{z-\lambda}(\Theta(z)-{\Theta}( \lambda))x\in K_{\Theta}\quad\text{for each }x\in\hsp.
\end{align*}
In particular,
$$\widetilde{\ke}_{0}^{\Theta}(z)x=\bar z(\Theta(z)-{\Theta}(0))x~ \quad \text{for each }x\in\hsp.$$
and
$$\widetilde{D}_{\Theta}f=\widetilde{\ke}_0^{\Theta}(\tau_{\Theta}f)(0)\in\widetilde{\mathcal{D}}_{\Theta},$$
where
$$\widetilde{\mathcal{D}}_{\Theta}=\tau_{\widetilde{\Theta}}\mathcal{D}_{\widetilde{\Theta}}=\{ \widetilde{\ke}_0^{{\Theta}}x: x\in \hsp\}=\{ \bar{z}(\Theta(z)-\Theta(0))x: x\in \hsp\}.$$
Observe that for $f\in K_{\Theta}$, $x\in\hsp$,
\begin{align*}
\langle f, \widetilde{\ke}_{\lambda}^{\Theta}x\rangle_{L^{2}(\hsp)}&=
\langle f,\tau_{\widetilde{\Theta}}(\ke_{\bar\lambda}^{\widetilde{\Theta}}x)\rangle_{L^{2}(\hsp)}\\
&=\langle \tau_{\Theta}f,\ke_{\bar\lambda}^{\widetilde{\Theta}}x\rangle_{L^{2}(\hsp)}=\langle (\tau_{\Theta}f)(\bar\lambda),x\rangle_{\hsp}.
\end{align*} Then it follows that for $f\in K_{\Theta}$ we have $M_zf\in K_{\Theta}$ if and only if $f\perp\widetilde{\mathcal{D}}_{\Theta}$.
Indeed, $M_zf\in K_\Theta$ if and only if $\Theta P_+(\Theta^{*}M_zf)=0$. Since
$$(\Theta^{*}M_zf)(z)=\Theta(z)^*zf(z)=(\tau_{\Theta}f)(\overline{z}),$$
we have $P_+(\Theta^{*}M_zf)=(\tau_{\Theta}f)(0)$ and so $M_{z}f\in K_{\Theta}$ if and only if
$$0=\langle(\tau_{\Theta}f)(0),x\rangle=\langle f, \widetilde{\ke}_{0}^{\Theta}x\rangle_{L^{2}(\hsp)}\quad \text{for every}\quad x\in \hsp,$$
i.e, $f\perp\widetilde{\mathcal{D}}_{\Theta}$. Therefore we get
$$(S_{\Theta}f)(z)=\left\{\begin{array}{cl}
 z f(z)&\text{for }f\perp\widetilde{\mathcal{D}}_{\Theta},\\
-\big(I_{\hsp}-\Theta(z)\Theta(0)^*\big)\Theta(0)x&\text{for }f=\widetilde{\ke}_{0}^{\Theta}x\in\widetilde{\mathcal{D}}_{\Theta}.
\end{array}\right.$$
Hence
$$\widetilde{{D}}_{\Theta}f=\left\{\begin{array}{cl}
0&\text{for }f\perp\widetilde{\mathcal{D}}_{\Theta},\\
\widetilde{\ke}_{0}^{\Theta}(I_{\hsp}-\Theta(0)\Theta(0)^*)x&\text{for }f=\widetilde{\ke}_{0}^{\Theta}x\in\widetilde{\mathcal{D}}_{\Theta}.
\end{array}\right.$$~\label{d}

Let us now consider conjugations. A conjugation $J$ on a Hilbert space
$\hsp$ is an antilinear map $J:\hsp\longrightarrow \hsp$ such that $J^2=I_\hsp$ and
$$\langle Jf,Jg\rangle=\langle g, f\rangle \quad \text{for all}\quad f,g\in \hsp $$
 Recall that a bounded linear operator $T:\hsp\longrightarrow\hsp$ is said to be $J$-symmetric ($J$ being a conjugation on $\hsp$) if $JTJ=T^*$. We say that $T$ is complex symmetric if it is $J$-symmetric with respect
to some conjugation $J$ (see \cite{GP}).

In \cite{CKLP} the authors consider certain classes of conjugations in $L^2(\hsp)$. One such conjugation is $\J^*:L^2(\hsp)\to L^2(\hsp)$ defined for a fixed conjugation $J$ on $\hsp$ by
\begin{equation}\label{J}
(\J^*f)(z)=J(f(\overline{z}))\quad \text{a.e. on}~~\mathbb{T}.
\end{equation}
It is not difficult to verify that for $f(z)=\sum_{n=-\infty}^{\infty}a_nz^n\in L^{2}(\oL)$, we have
$$(\J^*f)(z)=\sum_{n=-\infty}^{\infty}J(a_n)z^n.$$
Hence, $\J^*$ is an $M_z$-commuting conjugation, i.e, $\J^*M_z=M_z\J^*$,  $\J^*(H^{2}(\hsp))=H^{2}(\hsp)$, and $\J^*P_+=P_+\J^*$ (see \cite[Section 4]{CKLP}).

For $\Phi\in L^\infty(\oL)$ and an arbitrary conjugation $J$ on $\hsp$ let
\begin{equation}\label{F}
\Phi_J(z)=J\Phi(z)J\quad \text{a.e on}~~~ \mathbb{T}.
\end{equation}
Then $\Phi_J\in L^\infty(\oL)$. As observed in \cite{CKLP}, $\Phi_J\in H^\infty(\oL)$ if and only if $\Phi\in H^\infty(\oL)$, and $\Phi_J$ is an inner function if and only if $\Phi$ is. Clearly, $(\Phi_J)_J=\Phi.$ Let us also observe that if $\Phi$ is $J$-symmetric, that is, $J\Phi(z)J=\Phi(z)^*$ a.e on $\mathbb{T}$ (or equivalently $\Phi(\lambda)$ is $J$-symmetric for $\lambda$ in $\mathbb{D}$, see \cite{CKLP}), then $\Phi_J=\widetilde{\Phi}$.

For two conjugations $J_1$ and $J_2$ on $\hsp$, let $\J_1^*$ and $\J_2^*$ denote the corresponding conjugations on $L^2(\hsp)$ given by $\eqref{J}$. For each $f\in L^2(\hsp)$ we have
\begin{equation}\label{J2MF}
(\J_2^*M_{\Phi}f)(z)=J_2(\Phi(\overline{z})f(\overline{z}))=J_2\Phi(\overline{z})J_1J_1f(\overline{z})=(M_\G\J^*_1f)(z)
\end{equation}
where $\G(z)=J_2\Phi(\overline{z})J_1$ a.e. on $\mathbb{T}$. In particular
\begin{equation}\label{JM}
\J^*M_\Phi=M_\Phi\J^*
\end{equation}
(see \cite[Lemma 8.3]{CKLP}).

Note that $\Phi_J$ is also defined for $\Phi\in L^{2}(\oL)$. In that case \eqref{J2MF} and \eqref{JM} hold for $L^\infty(\hsp)$.

For $f\in L^2(\mathcal{H})$ we have the following equality:
\begin{equation}
(\J^*\J f)(z)=\J^*(\bar z f(\bar z))=J(z f(z))=\bar z J(f(z))
\end{equation}
and
\begin{equation}\label{J J^*: equality}
(\J \J^* f)(z)=\J(J f(\bar z))=\bar z J(f(z))=(\J^*\J f)(z).
\end{equation}

\smallskip
\begin{proposition}{\em \cite{CKLP}}\label{a,b,c,d}
Let $\Theta\in H^\infty(\oL)$ be a pure inner function and let $J$ be a conjugation on $\hsp$. Then the following identities hold.
\end{proposition}
\begin{enumerate}[(a)]
		\item $\J^*(\Theta H(\hsp))=\Theta_JH^2(\hsp)$;\label{a}
		\item $\J^*P_\Theta=P_{\Theta_J} \J^*$;\label{b}
\item $\J^*(K_\Theta)=K_{\Theta_J}$;\label{c}
\item $\J^*(\ke_{\lambda}^{\Theta}x)=\ke_{\overline{\lambda}}^{\Theta_{\J}}Jx$.\label{d}
\end{enumerate}

\smallskip
If $\Theta$ is $J$-symmetric, we obtain \cite[Proposition 7.7]{CKLP}:
\begin{proposition}{\em \cite{CKLP}}\label{1,2,3,4}
Let $J$ be a conjugation on $\hsp$ and let $\Theta$ be a $J$-symmetric inner function.
Then the following identities hold.
\end{proposition}
\begin{enumerate}[(a)]
		\item $\J^{*}M_{\Theta}=M_{\widetilde{\Theta}}\J^{*}$;\label{1}
		\item $\J^{*}(\Theta H^{2}(\hsp))=\widetilde{\Theta} H^{2}(\hsp)$;\label{2}
\item $\J^{*}(K_\Theta)=K_{\widetilde{\Theta}}$;\label{3}
\item $\J^{*}(\ke_{0}^{\Theta}x)=\ke_{0}^{\widetilde{\Theta}} J x $.\label{4}
\end{enumerate}
\smallskip
\begin{theorem}\label{4.4}
Let $\Theta_1,\Theta_2\in H^\infty(\oL)$ be two pure inner functions and let $J_1, J_2$ be two conjugations on $\hsp$. A bounded linear operator $B:K_{\Theta_1}\to K_{\Theta_2}$ belongs to $\mathcal{MH}(\Theta_1, \Theta_2)$ if and only if $\J_2^*B\J_1^*$ belongs to $\mathcal{MH}((\Theta_1)_{J_1}, (\Theta_2)_{J_2})$. More precisely, $B=B_\Phi^{\Theta_{1}, \Theta_2}\in \mathcal{MH}(\Theta_1, \Theta_2)$ if and only if
$\J_2^*B_{\Phi}\J_1^*=B_\Psi^{(\Theta_{1})_{J_1}, (\Theta_2)_{J_2}}\in \mathcal{MH}((\Theta_1)_{J_1}, (\Theta_2)_{J_2})$ with
\begin{equation}\label{equi}
\Psi(z)=J_2\Phi(\overline{z})J_1 \quad\text{a.e. on } ~~\mathbb{T}.
\end{equation}
\end{theorem}

\begin{proof}
Assume that $B=B_\Phi^{\Theta_{1}, \Theta_2}\in\mathcal{MH}(\Theta_1, \Theta_2)$ with $\Phi\in L^2(\oL)$. Let $f\in K_{(\Theta_1)_{J_1}}^\infty$. Note that $\J_1^*f\in K_{\Theta_1}^\infty$. Therefore, by Proposition \ref{a,b,c,d}\eqref{b}, \eqref{JM} and \eqref{J J^*: equality},
\begin{align*}
\J_2^*B\J_1^* f
&=\J_2^*P_{\Theta_2}\J(I-P)(\Phi\J_1^*f)\\
&=P_{(\Theta_2)_{J_2}}\J_2^*\J(\Phi\J_1^*f)\\
&=P_{(\Theta_2)_{J_2}}\J\J_2^*(\Phi\J_1^*f)\\
&=P_{(\Theta_2)_{J_2}}\J\J_2^*(\Phi J_1f(\bar z))\\
&=P_{(\Theta_2)_{J_2}}\J(J_2\Phi(\bar z) J_1f( z))\\
&=P_{(\Theta_2)_{J_2}}\J(I-P_{+})(J_2\Phi(\bar z) J_1f( z))\\
&=B_{\Psi}^{(\Theta_1)_{J_1},(\Theta_2)_{J_2}}f
\end{align*}
where $\Psi$ is given by $\eqref{equi}$. Thus $\J_2^*B\J_1^*\in \mathcal{MH}((\Theta_1)_{J_1}, (\Theta_2)_{J_2})$.

 On the other hand, if
$B=\J_2^*B_\Psi^{(\Theta_{1})_{J_1}, (\Theta_2)_{J_2}}\J_1^*\in\mathcal{MH}((\Theta_1)_{J_1}, (\Theta_2)_{J_2})$  with some $\Psi\in L^2(\oL)$, then $B=\J_2^*B_{\Psi}^{(\Theta_1)_{J_1},(\Theta_2)_{J_2}}\J_1^* $ and as above, $B=B_{\Phi}^{\Theta_1,\Theta_2}$ with
$$\Phi(z)=J_2\Psi(\bar z)J_1\quad \text{a.e. on}~~ \mathbb{T}.$$
\end{proof}
For the scalar case, Theorem \ref{4.4} can be found in \cite{BM} (see also \cite{cg} for the symmetric case).

In the scalar case each model space $K_{\theta}$ is equipped with a natural conjugation $C_{\theta}$ defined in terms of boundary functions by
$(C_\theta f)(z)=\theta(z)\overline{z}\overline{f(z)}$. If $\Theta\in H^\infty(\oL)$ is an inner function and $J$ is a conjugation on $\hsp$ we can similarly define $\C_{\Theta}:L^2(\hsp)\to L^2(\hsp)$ by
$$(\C_{\Theta}f)(z)=\Theta(z)\overline{z} J(f(z))\quad \text{a.e. on }~~\mathbb{T}.$$
Although $\C_{\Theta}$ is obviously an antilinear isometry, it is not in general an involution. A simple computation shows that $\C_{\Theta}$ is an involution (and so a conjugation) if and only if $\Theta(z)J\Theta(z)J=I_{\hsp}$ a.e. on $\mathbb{T}$, i.e., if and only if $\Theta$ is $J$-symmetric.

If $\Theta$ is $J$-symmetric, then $\C_{\Theta}(\Theta H^2(\hsp))=H^2(\hsp)^{\perp}$ and so
$$\C_{\Theta}(K_\Theta)=K_\Theta.$$
Note that in that case
$$\C_{\Theta}=\J^*\tau_\Theta.$$
By Theorem \ref{ttau} and Theorem \ref{4.4} we get the following.

\begin{theorem}\label{4.5}
Let $\Theta_1, \Theta_2\in H^{\infty}(\oL)$ be two pure inner functions and let $J_1, J_2$ be two conjugations on $\hsp$ such that $\Theta_1$ is $J_1$-symmetric and $\Theta_2$ is $J_2$-symmetric. A bounded linear operator $B:K_{\Theta_1}\to K_{\Theta_2}$ belongs to $\mathcal{MH}(\Theta_1, \Theta_2)$ if and only if $C_{\Theta_2}BC_{\Theta_1}$ belongs to $\mathcal{MH}(\Theta_1, \Theta_2)$. More precisely, $B=B_{\Phi}^{\Theta_1, \Theta_2}\in \mathcal{MH}(\Theta_1, \Theta_2)$ if and only if $$C_{\Theta_2}B_{\Phi}^{\Theta_1, \Theta_2}C_{\Theta_1}=B_{\Psi}^{\Theta_1,\Theta_2}\in \mathcal{MH}(\Theta_1, \Theta_2)$$ with
\begin{equation}
\Psi(z)=J_2\Theta_{2}(z)\Phi(z)\Theta_{1}(\bar z)^{*}J_1\quad \text{a.e. on } \mathbb{T}
\end{equation}
\end{theorem}
\begin{proof}
Let $\Phi\in L^{2}({\oL})$ and $B_{\Phi}^{\Theta_1, \Theta_2}\in \mathcal{MH}(\Theta_1, \Theta_2)$. Consider
\begin{align*}
C_{\Theta_2}B_{\Phi}^{\Theta_1, \Theta_2}C_{\Theta_1}
&=C_{\Theta_2}B_{\Phi}^{\Theta_1, \Theta_2}C_{\Theta_1}^{*}=\J_{2}^{*}\tau_{\Theta_2}B_{\Phi}^{\Theta_1, \Theta_2}\tau^{*}_{\Theta_1}\J_{1}^{*}=\J_{2}^{*}B_{\chi}^{\widetilde{\Theta}_{1}, \widetilde{\Theta}_{2}}\J^{*}_{1}
\end{align*}
where the symbol $\chi(z)$ is given by (as Theorem \ref{ttau})
$$\chi(z)=\widetilde{\Theta}_{2}(\z)^{*}\Phi(\bar z)\widetilde{\widetilde{\Theta}}_{1}(z)^{*}=\Theta_{2}(\bar z)\Phi(\bar z)\Theta_{1}(z)^{*}$$
and
\begin{align*}
C_{\Theta_2}B_{\Phi}^{\Theta_1, \Theta_2}C_{\Theta_1}=\J_{2}^{*}B_{\chi}^{\widetilde{\Theta}_{1}, \widetilde{\Theta}_{2}}\J^{*}_{1}=B_{\Psi}^{(\widetilde{\Theta}_{1})_{J_1}, (\widetilde{\Theta}_{2})_{J_2}}=B_{\Psi}^{\Theta_1, \Theta_2}
\end{align*}
with the symbol $\Psi$ is given by
$$\Psi(z)=J_2\chi(\bar z) J_1=J_2\Theta_{2}(z)\Phi(z)\Theta_{1}(\bar z)^{*}J_1$$ from
Theorem \ref{4.4}.
\end{proof}

\smallskip

For the scalar version of Theorem \ref{4.5}, please see \cite{BM}.
\begin{remark}
Recall that in the scalar case $\hsp=\mathbb{C}$ every TTO on the model space $K_\theta$ is $C_\theta$-symmetric, i.e.,
$$C_\theta A_\varphi^\theta C_\theta=(A_\varphi^\theta)^*=A_{\overline{\varphi}}^\theta$$
(see, e.g., \cite{sar}). In that case however the only conjugation in $\hsp$ we need to consider is $J(z)=\overline{z}$ (and each $\varphi\in L^2$ is $J$-symmetric). In the vector valued case, the equality
\begin{equation}\label{eq: 4.12}
C_{\Theta} A_{\Phi}^\theta C_{\Theta}=A_{\Phi^*}^\Theta.
\end{equation}
is not necessarily true for an arbitrary $\Phi\in L^2(\oL)$ (even though we assume here that $\Theta$ is $J$-symmetric). It is however satisfied if also $\Phi$ is $J$-symmetric and commute with $\Theta$ (see \cite{KT}).
\end{remark}

\medskip

\section{Characterizations of Matrix valued asymmetric truncated Toeplitz operators}\label{s5}

In \cite[Theorem 5.2 and Remark 5.4]{KT} characterizations of matrix valued truncated Toeplitz operators in
$\mathcal{MT}(\Theta)$ were given by using the model operators $S_{\Theta}$, $S_{\Theta}^*$ and the defect operators $D_{\Theta}$, $\widetilde{D}_{\Theta}$. These characterizations generalized D. Sarason's results for the scalar case (see \cite{sar}). We obtain analogous results for matrix valued asymmetric truncated Toeplitz operators in \cite{RK3}.

\subsection{Characterizations by Compressed Shift}
\begin{lemma}{\em\cite{RK3}}\label{10}
	If $\Phi\in H^2(\mathcal{L}(\hsp))$, then
	$$A_\Phi^{{\Theta_1},{\Theta_2}}-S_{\Theta_2}A_\Phi^{{\Theta_1},{\Theta_2}}S_{\Theta_1}^*=P_{\Theta_2}M_{\Phi}(I_{H^2(\hsp)}-SS^*)\quad\text{on }K_{\Theta_1}^{\infty}.$$
\end{lemma}
\smallskip

Recall that
$$D_{\Theta}= I_{K_{\Theta}} - S_{\Theta}S_{\Theta}^*,\quad\widetilde{D}_{\Theta}= I_{K_{\Theta}} - S_{\Theta}^*S_{\Theta}$$
and
$$\mathcal{D}_{\Theta}=\{ (I_{\hsp}-\Theta(z)\Theta(0)^*)x : x\in \hsp\},\quad \widetilde{\mathcal{D}}_{\Theta}=\{ \bar{z}(\Theta(z)-\Theta(0))x: x\in \hsp \},$$
while the operator $\Omega_{\Theta}:\mathcal{D}_{\Theta}\rightarrow \hsp\subset H^2(\hsp)$ is defined by  $$\Omega_{\Theta}(\ke_0^{\Theta}x)=x.$$
\smallskip
\begin{theorem}{\em\cite{RK3}}\label{tcara}
	Let $ {\Theta_1},{\Theta_2} \in H^{\infty}(\oL)$ be two pure inner functions and let $A:K_{\Theta_{1}}\rightarrow  K_{\Theta_{2}}$ be a bounded linear operator. Then $A$ belongs to $\mathcal{MT}(\Theta_{1}, \Theta_{2})$  if and only if there exist bounded linear operators \mbox{$B_1:\mathcal{D}_{\Theta_{1}}\to K_{\Theta_{2}}$} and $B_2: \mathcal{D}_{\Theta_{2}}\to K_{\Theta_{1}}$, such that
	\begin{equation}\label{cara}
	A-S_{\Theta_2}AS_{\Theta_1}^*=B_1\dfc{1}+\dfc{2}B_2^*.
	\end{equation}
\end{theorem}

\smallskip
If a bounded operator $A: K_{\Theta_1}\to K_{\Theta_2}$ satisfies \eqref{cara}, then $A=A_{\Psi+\Xi^*}^{\Theta_1, \Theta_2}\in \mathcal{MT}(\Theta_1, \Theta_2)$ with $\Psi,\Xi\in H^2(\oL)$ are given below.
\begin{corollary}{\em\cite{RK3}}\label{cep}
	Let $ {\Theta_1},{\Theta_2} \in H^{\infty}(\oL)$ be two pure inner functions and let $A:K_{\Theta_{1}}\rightarrow  K_{\Theta_{2}}$ be a bounded linear operator.
\end{corollary}
\begin{enumerate}[(a)]
		\item If $A$ satisfies \eqref{cara}, then $A=A_{\Psi+\Xi^*}^{{\Theta_1},{\Theta_2}}\in \mathcal{MT}(\Theta_{1}, \Theta_{2})$ with
			\begin{equation}\label{symm} \Psi(z)x=\big(B_1\ke_{0}^{\Theta_{1}}x\big)(z)\quad\text{and}\quad\Xi(z)x=\big(B_2\ke_{0}^{\Theta_{2}}x\big)(z)~\mbox{for}~ x\in \hsp.
			\end{equation}
		\item If $A=A_{\Psi+\Xi^*}^{{\Theta_1},{\Theta_2}}\in \mathcal{MT}(\Theta_{1}, \Theta_{2})$, then $A$ satisfies \eqref{cara} with
\begin{equation}\label{symmm} B_1=P_{\Theta_2}M_{\Psi}\Omega_{\Theta_1}\quad\text{and}\quad B_2=P_{\Theta_1}M_{\Xi}\Omega_{\Theta_2}.
\end{equation}
\end{enumerate}
\smallskip
\begin{remark}\label{rmk}
\begin{enumerate}[(a)]
\item For an inner function $\Theta\in H^{\infty}(\oL)$ denote
$$\mathcal{M}_\Theta=H^2(\oL)\ominus \Theta H^2(\oL).$$
Therefore, if a bounded linear operator $A:K_{\Theta_1}\to K_{\Theta_2}$ satisfies \eqref{cara}, then $A=A_{\Psi+\Xi^*}^{\Theta_1, \Theta_2}\in \mathcal{MT}(\Theta_1, \Theta_2)$ with $\Psi\in \mathcal{M}_{\Theta_2}$ and $\Xi\in \mathcal{M}_{\Theta_1}$ given by \eqref{symm}.
\item Recall that $A_{\Phi}^{\Theta_1, \Theta_2}=0$ if and only if
$$\Phi\in \Theta_2 H^2(\oL)+(\Theta_1 H^2(\oL))^*~(\mbox{see  \cite{RK2}}).$$
Clearly, if $A=A_{\Psi+\Xi^*}^{\Theta_1, \Theta_2}$ with $\Psi, \Xi\in H^2(\oL)$, then the operators $B_1$ and $B_2$ given by \eqref{symmm} do not depend on the parts of $\Psi$ and $\Xi$ that belong to $\Theta_2 H^2(\oL)$ and $\Theta_1 H^2(\oL)$, respectively.
\end{enumerate}
\end{remark}
As in \cite{KT} we can use the unitary operator $\tau_\Theta$ defined by \eqref{tau} and obtain the following theorem.

\smallskip
\begin{theorem}{\em\cite{RK3}}\label{tca2}
	Let $ {\Theta_1},{\Theta_2} \in H^{\infty}(\oL)$ be two pure inner functions and let $A:K_{\Theta_{1}}\rightarrow  K_{\Theta_{2}}$ be a bounded linear operator. Then $A$ belongs to $\mathcal{MT}(\Theta_{1}, \Theta_{2})$  if and only if there exist bounded linear operators \mbox{$\widetilde{B}_1:\widetilde{\mathcal{D}}_{\Theta_{1}}\rightarrow K_{\Theta_{2}}$} and $\widetilde{B}_2: \widetilde{\mathcal{D}}_{\Theta_{2}}\rightarrow K_{\Theta_{1}}$, such that
	\begin{equation}\label{cara2}
	A-S_{\Theta_2}^*AS_{\Theta_1}=\widetilde{B}_1\dcf{1}+\dcf{2}\widetilde{B}_2^*.
	\end{equation}
\end{theorem}
\smallskip
Note from the proof of Theorem \ref{tca2} that if $A:K_{\Theta_1}\to K_{\Theta_2}$ satisfies \eqref{cara2} with some $\widetilde{B}_1:\widetilde{\mathcal{D}}_{\Theta_1}\to K_{\Theta_2}$ and $\widetilde{B}_2:\widetilde{\mathcal{D}}_{\Theta_2}\to K_{\Theta_1}$ where $B_1$ and $B_2$ are given by \eqref{symmm}, then
$$\widetilde{A}=\tau_{\Theta_2}A\tau_{\Theta_1}^*.$$ By Corollary \ref{cep}, $\widetilde{A}=A_{\Psi+\Xi^*}^{\widetilde{\Theta}_1, \widetilde{\Theta}_2}$ with
$$\Psi(z)x=(B_1\ke_{0}^{\widetilde{\Theta}_1}x)(z)
=(\tau_{\Theta_2}\widetilde{B}_1\tau^*_{\Theta_1}\ke_{0}^{\widetilde{\Theta}_1}x)(z)=(\tau_{\Theta_2}\widetilde{B}_1\widetilde{\ke}_{0}^{\Theta_1}x)(z)$$
and
$$\Xi(z)x=(B_2\ke_{0}^{\widetilde{\Theta}_1}x)(z)=(\tau_{\Theta_1}\widetilde{B}_2\tau^*_{\Theta_2}\ke_{0}^{\widetilde{\Theta}_2}x)(z)
=(\tau_{\Theta_1}\widetilde{B}_2\widetilde{\ke}_{0}^{\Theta_2}x)(z).$$
Moreover (see Remark \ref{rmk}), $\Psi\in \mathcal{M}_{\widetilde{\Theta}_2}$ and $\Xi\in \mathcal{M}_{\widetilde{\Theta}_1}$.

\smallskip

As in the scalar case, we can use Theorem \ref{tcara} and Theorem \ref{tca2} to get the following.
\begin{corollary}{\em\cite{RK3}}
Let $\Theta_1, \Theta_2\in H^\infty(\oL)$ be two pure inner functions and let $A:K_{\Theta_1}\to K_{\Theta_2}$ be a bounded linear operator. Then $A$ belongs to $\mathcal{MT}(\Theta_1, \Theta_2)$ if and only if the following hold:
\end{corollary}
\begin{enumerate}[(a)]
\item there exist bounded linear operators $\widehat{B}_1:\mathcal{D}_{\Theta_1}\to K_{\Theta_2}$ and $\widehat{B}_2:\widetilde{{\mathcal{D}}}_{\Theta_2}\to K_{\Theta_1}$ such that
    $$S_{\Theta_2}^*A-AS_{\Theta_1}^*=\widehat{B}_1D_{\Theta_1}+\widetilde{D}_{\Theta_2}\widehat{B}_2^*.$$\label{i}
    \item there exist bounded linear operators $\widehat{B}_1:\widetilde{\mathcal{D}}_{\Theta_1}\to K_{\Theta_2}$ and $\widehat{B}_2:\mathcal{D}_{\Theta_2}\to K_{\Theta_1}$ such that
        $$S_{\Theta_2}A-AS_{\Theta_1}=\widehat{B}_1\widetilde{D}_{\Theta_1}+D_{\Theta_2}\widehat{B}_{2}^*.$$\label{ii}
\end{enumerate}

\subsection{By Shift invariance}
In the scalar case the notion of shift invariance for TTO's was introduced in \cite{sar}. D. Sarason proved that a bounded linear operators $A:K_\theta\to K_\theta$ is a TTO if and only if it is shift invariant, i.e.,
$$\langle ASf, Sf\rangle_{L^2}=\langle Af, f\rangle_{L^2}$$
 for each $f\in K_\theta$ such that $Sf\in K_\theta.$
In \cite{KT} we prove that the same is true for MATTO's.

For ATTO's the notion of shift invariance was introduced in \cite{CKP2} ( see also \cite{l2}). Shift invariance of matrix valued truncated Toeplitz operators is presented in \cite{KT}. Here we consider shift invariance of MATTO's. As in the scalar case (see \cite{BM}), we characterize MATTO's in term of four (equivalent) types of shift invariance.

Recall that for an operator valued inner function $\Theta\in H^\infty(\oL)$ and for $f\in K_{\Theta}$ we have
$$Sf=M_zf\in K_{\Theta}~~\text{if and only if} ~~ f\perp \widetilde{\mathcal{D}}_{\Theta}~~ ((\tau_\Theta f(0))=0)$$
and
$$S^*f=M_{\overline{z}}f\in K_{\Theta} ~~ \text{if and only if} ~~ f\perp\mathcal{D}_{\Theta}~~ (f(0)=0).$$
\begin{theorem}{\em\cite{RK3}}
Let $\Theta_1, \Theta_2\in H^\infty(\oL)$ be two pure inner functions and let $A:K_{\Theta_1}\to K_{\Theta_2}$ be a bounded linear operator. Then $A$ belongs to $\mathcal{MT}(\Theta_1, \Theta_2)$ if and only if it has one {\em(}and all{\em)} of the following properties:
\end{theorem}
\begin{enumerate}[(a)]
\item $\langle AS^*f, S^*g\rangle_{L^2(\hsp)}=\langle A f,g\rangle_{L^2(\hsp)}$ for all $f\in K_{\Theta_1}$, $g\in K_{\Theta_2}$ such that $f\perp\mathcal{D}_{\Theta_1}$, $g\perp\mathcal{D}_{\Theta_2}$;\label{T1}
    \item $\langle AS^*f, g\rangle_{L^2(\hsp)}=\langle Af,Sg\rangle_{L^2(\hsp)}$ for all $f\in K_{\Theta_1}$, $g\in K_{\Theta_2}$ such that $f\perp\mathcal{D}_{\Theta_1}$, $g\perp\widetilde{\mathcal{D}}_{\Theta_2}$;\label{T2}
    \item $\langle ASf, Sg\rangle_{L^2(\hsp)}=\langle Af,g\rangle_{L^2(\hsp)}$ for all $f\in K_{\Theta_1}$, $g\in K_{\Theta_2}$ such that $f\perp\widetilde{\mathcal{D}}_{\Theta_1}$, $g\perp\widetilde{\mathcal{D}}_{\Theta_2}$;\label{T3}
    \item $\langle ASf, g\rangle_{L^2(\hsp)}=\langle Af,S^*g\rangle_{L^2(\hsp)}$ for all $f\in K_{\Theta_1}$, $g\in K_{\Theta_2}$ such that $f\perp\widetilde{\mathcal{D}}_{\Theta_1}$, $g\perp\mathcal{D}_{\Theta_2}$.\label{T4}
\end{enumerate}

\subsection{Characterization by modified compressed shift}\label{char by of MTTO by mod comp shift}
The notion of modified compressed shift was introduced in \cite{KT}. For any nonconstant inner function $\Theta$, suppose that $X_\Theta : \widetilde{\mathcal{D}}_{\Theta}\rightarrow \mathcal{D}_{\Theta}$, and consider $\widehat{X}_\Theta\in \mathcal{ L}(K_\Theta)$ defined by $\widehat{X}_\Theta f=X_\Theta P_{\widetilde{\mathcal{D}}_{\Theta}}f$.
The operator modified shift is  defined by
$$S_{\Theta, X_\Theta}=S_{\Theta}+(\widehat{X}_\Theta-S_{\Theta})P_{\widetilde{\mathcal{D}}_{\Theta}}, $$
or $$S_{\Theta, X_\Theta}=S_{\Theta}+P_{\mathcal{D}_{\Theta}}Y_{\Theta}P_{\widetilde{\mathcal{D}}_{\Theta}}, $$
implies that $$S_{\Theta}=S_{\Theta, X_\Theta}-P_{\mathcal{D}_{\Theta}}Y_{\Theta}P_{\widetilde{\mathcal{D}}_{\Theta}}$$
where $Y_{\Theta}=\widehat{X}_{\Theta}-S_{\Theta}$.

From equation (3.8) of \cite{KT}, there is an operator $J_{\Theta_{1}}\in \mathcal{L}(K_{\Theta_{1}})$ such that
\begin{equation}\label{J1D1}
P_{\mathcal{D}_{\Theta_{1}}}=(I-S_{\Theta_{1}}S_{\Theta_{1}}^*)J_{\Theta_{1}}=D_{\Theta_{1}}J_{\Theta_{1}}=J_{\Theta_{1}}^*D_{\Theta_{1}},
\end{equation}
  and similarly there is $J_{\Theta_{2}}\in \mathcal{L}(K_{\Theta_{2}})$ such that
  \begin{equation}\label{J2D2}
  P_{\mathcal{D}_{\Theta_{2}}}=(I-S_{\Theta_{2}}S_{\Theta_{2}}^*)J_{\Theta_{2}}=D_{\Theta_{2}}J_{\Theta_{2}}=J_{\Theta_{2}}^*D_{\Theta_{2}}.
  \end{equation}
\smallskip
\begin{theorem}{\em\cite{RK3}} Let $ {\Theta_1}, {\Theta_2}\in H^\infty (\oL)$ be two pure inner functions and let $A:K_{\Theta_{1}}\rightarrow  K_{\Theta_{2}}$ be a bounded operator.
Then $A\in\mathcal{MT}(\Theta_{1}, \Theta_{2})$  if and only if
\begin{align*}
A-S_{\Theta_2, X_{\Theta_2}}AS_{\Theta_1, X_{\Theta_1}}^*=BD_{\Theta_{1}}+D_{\Theta_{2}}B^{'*},
\end{align*}
with $B_1:\mathcal{D}_{\Theta_{1}}\to K_{\Theta_{2}}$ and $B_{2}^{\prime}:\mathcal{D}_{\Theta_{2}}\to K_{\Theta_1}$.
\end{theorem}

\section{Connection between MATTO's and MATHO's}
Let $J_1, J_2$ be conjugations on $\hsp$ such that $\Theta_1$ is $J_1$-symmetric and $\Theta_2$ is $J_2$-symmetric. Then we have the following relations between MATTO's and MATHO's.
\begin{proposition}\label{Prop: a,b,c,d,e,f}
For every $\Phi\in L^{2}(\mathcal{L}(\hsp))$, we have
\item[(a)] $C_{\Theta_2}A_{\Phi}^{\Theta_{1}, \Theta_{2}}C_{\Theta_1}=A^{\Theta_{1}, \Theta_{2}}_{J_2\Theta_{2}^{*} \Phi \Theta_{1}J_1}$.
\item[(b)] $C_{\Theta_2}B_{\Phi}^{\Theta_{1}, \Theta_{2}}C_{\Theta_1}=B^{\Theta_{1}, \Theta_{2}}_{J_2\widetilde{\Theta}_{2} \Phi \Theta_{1}J_1}$.
\item[(c)] $\J_{2}^{*}A_{\Phi}^{\Theta_{1}, \Theta_{2}}\J_{1}^{*}=A_{J_2\Phi(\bar z)J_1}^{\widetilde{\Theta}_{1}, \widetilde{\Theta}_{2}}$.
\item[(d)] $\J_{2}^{*}B_{\Phi}^{\Theta_{1}, \Theta_{2}}\J_{1}^{*}=B_{J_2\Phi(\bar z)J_1}^{\widetilde{\Theta}_{1}, \widetilde{\Theta}_{2}}$.
\item[(e)] $C_{\Theta_2}A_{\Phi}^{\Theta_{1}, \Theta_{2}}\J_{1}^{*}=B_{J_{2}\Theta_{2}(\bar z)^{*}\Phi(\bar z) J_{1}}^{\widetilde{\Theta}_{1},\Theta_{2}}$.
\item[(f)] $\J_{2}^{*}B_{\Phi}^{\Theta_{1}, \Theta_{2}}C_{\Theta_1}=A_{J_{2}\Phi \Theta_1 J_{1}}^{\Theta_{1}, \widetilde{\Theta}_{2}}$.
\end{proposition}
\begin{proof}
(a)  Let $f\in K_{\Theta_{1}}$ and $g\in K_{\Theta_{2}}$ and  let $J_1 \Theta_{1}J_1=\Theta_{1}^{*}$ and $J_2\Theta_{2}J_2=\Theta_{2}^{*}$. Then
\begin{align*}
\langle C_{\Theta_2}A_{\Phi}^{\Theta_{1}, \Theta_{2}}C_{\Theta_1}f, g\rangle
&=\langle C_{\Theta_2}g, A_{\Phi}^{\Theta_{1}, \Theta_{2}}C_{\Theta_1}f\rangle
=\langle \Theta_{2}\bar z J_2 g, \Phi \Theta_{1}\bar zJ_1 f\rangle\\
&=\langle \Theta_{2}J_2 g, \Phi \Theta_{1}J_1 f\rangle
=\langle J_2 g, \Theta_{2}^{*} \Phi \Theta_{1}J_1 f\rangle\\
&=\langle J_2\Theta_{2}^{*} \Phi \Theta_{1}J_1 f,g \rangle=\langle A_{J_2\Theta_{2}^{*} \Phi \Theta_{1}J_1} f,g \rangle.
\end{align*}
(b) Let $f\in K_{\Theta_{1}}$ and $g\in K_{\Theta_{2}}$. Then
\begin{align*}
\langle C_{\Theta_2}B_{\Phi}^{\Theta_{1}, \Theta_{2}}C_{\Theta_1}f, g\rangle
&=\langle C_{\Theta_2}g, B_{\Phi}^{\Theta_{1}, \Theta_{2}}C_{\Theta_1}f\rangle
=\langle C_{\Theta_2}g, \J(I-P)(\Phi C_{\Theta_1}f)\rangle\\
&=\langle \J C_{\Theta_2}g, \Phi C_{\Theta_1}f\rangle
=\langle \J(\Theta_{2}\bar z J_2 g), \Phi \Theta_{1}\bar z J_1 f\rangle\\
&=\langle \bar z\Theta_{2}(\bar z) z J_2( g(\bar z)), \bar z\Phi \Theta_{1}J_1 f\rangle\\
&=\langle  J_2\J g, \Theta_{2}(\bar z)^{*} \Phi \Theta_{1}J_1 f\rangle
=\langle J_2\widetilde{\Theta}_{2} \Phi \Theta_{1}J_1 f, \J g \rangle\\
&=\langle P_{\Theta_2}\J(I-P)(J_2\widetilde{\Theta}_{2} \Phi \Theta_{1}J_1 f),g \rangle\\
&=\langle B_{J_2\widetilde{\Theta}_{2} \Phi \Theta_{1}J_1}^{\Theta_{1}, \Theta_{2}} f,g \rangle.\\
\end{align*}
\noindent (c) For $f\in K_{\widetilde{\Theta}_{1}}$ and $g\in K_{\widetilde{\Theta}_{2}}$, it holds that
\begin{align*}
\langle \J_{2}^{*} A_{\Phi}^{\Theta_{1}, \Theta_{2}}\J_{1}^{*} f, g\rangle
&=\langle \J^{*}_{2} g, A_{\Phi}^{\Theta_{1}, \Theta_{2}}\J_{1}^{*}f\rangle
=\langle \J^{*}_{2} g, \Phi \J_{1}^{*} f\rangle\\
&=\langle J_2 g(\bar z), \Phi J_1 f(\bar z)\rangle
=\langle J_2 \Phi J_1 f(\bar z), g(\bar z)\rangle\cr
&=\langle J_2\Phi(\bar z)J_1f, g \rangle=\langle A_{J_2\Phi(\bar z)J_1}^{\widetilde{\Theta}_{1}, \widetilde{\Theta}_{2}} f, g \rangle.
\end{align*}
(d) Let $f\in K_{\widetilde{\Theta}_{1}}$ and $g\in K_{\widetilde{\Theta}_{2}}$. Then
\begin{align*}
\langle \J_{2}^{*}B_{\Phi}^{\Theta_{1}, \Theta_{2}}\J_{1}^{*}f, g\rangle
&=\langle \J_{2}^{*}g, B_{\Phi}^{\Theta_{1}, \Theta_{2}}\J_{1}^{*}f\rangle
=\langle \J_{2}^{*} g, \J\Phi \J_{1}^{*} f\rangle\\
&=\langle \J\J_{2}^{*}g, \Phi \J_{1}^{*} f\rangle
=\langle \J_{2}^{*}\J g, \Phi \J_{1}^{*} f\rangle\\
&=\langle \J_{2}^{*}\Phi \J_{1}^{*} f, \J g\rangle
=\langle J_2 \Phi(\bar z)J_1 f(z), \J g \rangle\\
&=\langle \J (I-P_{+})(J_2 \Phi(\bar z)J_1 f(z)), g \rangle\\
&=\langle B_{J_2 \Phi(\bar z)J_1}^{\widetilde{\Theta}_{1}, \widetilde{\Theta}_{2}} f, g \rangle.
\end{align*}
(e) For $f\in K_{\widetilde{\Theta}_{1}}$ and $g\in K_{\Theta_{2}}$, it holds that
\begin{align*}
\langle C_{\Theta_2}A_{\Phi}^{\Theta_{1}, \Theta_{2}}\J_{1}^{*}f, g\rangle
&=\langle C_{\Theta_2}g, A_{\Phi}^{\Theta_{1}, \Theta_{2}}\J_{1}^{*}f\rangle
=\langle \Theta_{2}\bar z J_{2} g, \Phi \J_{1}^{*} f\rangle\\
&=\langle \Theta_{2}\J^{*}J g, \Phi \J^{*} f\rangle
=\langle f, \J_{1}^{*}\Phi^{*}\Theta_{2}\bar z J_{2} g(z) \rangle\\
&=\langle f, J_{1}\Phi(\bar z)^{*}\Theta_{2}(\bar z) \z J_{2} g(\bar z) \rangle\\
&= \langle f, J_{1}\Phi(\bar z)^{*}\Theta_{2}(\bar z) J_{2}(\bar z  g(\bar z)) \rangle\\
&=\langle f, J_{1}\Phi(\bar z)^{*}\Theta_{2}(\bar z) J_{2} \J g \rangle\\
&=\langle J_{2}\Theta_{2}(\bar z)^{*}\Phi(\bar z) J_{1}f,  \bar z g(\bar z) \rangle\\
&=\langle z J_{2}\Theta_{2}(\bar z)^{*}\Phi(\bar z) J_{1}f, g(\bar z) \rangle\\
&=\langle \bar z J_{2}\Theta_{2}( z)^{*}\Phi( z) J_{1}f(\bar z), g(z) \rangle\\
&=\langle \J J_{2}\Theta_{2}(\bar z)^{*}\Phi(\bar z) J_{1}f(z), g(z) \rangle\\
&=\langle \J(I-P_{+}) (J_{2}\Theta_{2}(\bar z)^{*}\Phi(\bar z) J_{1}f), g \rangle\\
&=\langle B_{J_{2}\Theta_{2}(\bar z)^{*}\Phi(\bar z) J_{1}}^{\widetilde{\Theta}_{1}, \Theta_{2}}f, g \rangle.
\end{align*}

(f)
Let $f\in K_{\Theta_{1}}$ and $g\in K_{\widetilde{\Theta}_{2}}$. Then
\begin{align*}
\langle \J_{2}^{*}B_{\Phi}^{\Theta_{1}, \Theta_{2}}C_{\Theta_1}f, g\rangle
&=\langle \J_{2}^{*}g, B_{\Phi}^{\Theta_{1}, \Theta_{2}}C_{\Theta_1}f\rangle
=\langle J_{2} g(\bar z), \J\Phi \Theta_1\bar z J_{1}f\rangle\\
&=\langle J_{2} g(\bar z), \bar z\Phi(\bar z) \Theta_1(\bar z) z J_{1}f(\bar z)\rangle\\
&=\langle J_{2}\Phi(\bar z) \Theta_1(\bar z)  J_{1}f(\bar z) , g(\bar z)\rangle\\
&=\langle J_{2}\Phi(z) \Theta_1(z)  J_{1}f , g\rangle
=\langle P_{\widetilde{\Theta}_{2}}(J_{2}\Phi \Theta_1 J_{1}f) , g\rangle\\
&=\langle A_{J_{2}\Phi \Theta_1 J_{1}}^{\Theta_{1}, \widetilde{\Theta}_{2}} f, g\rangle.
\end{align*}
\end{proof}
\smallskip
\begin{corollary}\label{coro: a,b,c,d,e,f}
Let $A$ and $B$ be bounded linear operators from $K_{\Theta_1}$ into $K_{\Theta_2}$ for some inner functions $\Theta_1$ and $\Theta_2$. Then
\item[(a)] $A\in \mathcal{MT}(\Theta_1, \Theta_2)$ if and only if $C_{\Theta_2}AC_{\Theta_1}\in \mathcal{MT}(\Theta_1, \Theta_2)$.
\item[(b)] $B\in \mathcal{MH}(\Theta_1, \Theta_2)$ if and only if $C_{\Theta_2}BC_{\Theta_1}\in \mathcal{MH}(\Theta_1, \Theta_2)$.
\item[(c)] $A\in \mathcal{MT}(\Theta_1, \Theta_2)$ if and only if $\J_{2}^{*}A\J_{1}^{*}\in \mathcal{MT}(\widetilde{\Theta}_1, \widetilde{\Theta}_2)$.
\item[(d)] $B\in \mathcal{MH}(\Theta_1, \Theta_2)$ if and only if $\J_{2}^{*}B\J_{1}^{*}\in \mathcal{MH}(\widetilde{\Theta}_1, \widetilde{\Theta}_2)$.
\item[(e)] $A\in \mathcal{MT}(\Theta_1, \Theta_2)$ if and only if $C_{\Theta_2}A\J_{1}^{*}\in \mathcal{MH}(\widetilde{\Theta}_1, \Theta_2)$.
\item[(f)] $B\in \mathcal{MH}(\Theta_1, \Theta_2)$ if and only if $\J_{2}^{*}B C_{\Theta_1}\in \mathcal{MT}(\Theta_1, \widetilde{\Theta}_2)$.
\end{corollary}

\begin{proof}
The proof follows from Proposition \ref{Prop: a,b,c,d,e,f}.
\end{proof}
\smallskip
The connection between MATTO's and MATHO's makes it easy to to transform the results from MATTO's to MATHO's. The following proposition is a consequence of this connection.
\begin{proposition}
Let $\Theta_1$ and $\Theta_2$ be two nonconstant inner functions and let $\Phi\in L^2(\oL)$. Then $B_{\Phi}^{\Theta_1, \Theta_2}=0$ if and only if $\Phi\in J_2 [H^{2}(\oL)]^* J_1+J_2\widetilde{\Theta}_2 H^2(\oL)\Theta_1J_1$.
\end{proposition}

\begin{proof}
Proposition \ref{Prop: a,b,c,d,e,f}$(f)$ implies that $B_{\Phi}^{\Theta_1, \Theta_2}=0$ if and only if $A_{J_{2}\Phi \Theta_1 J_{1}}^{\Theta_{1}, \widetilde{\Theta}_{2}}=0$. By \cite[Theorem 4.2]{RK2}, $A_{J_{2}\Phi \Theta_1 J_{1}}^{\Theta_{1}, \widetilde{\Theta}_{2}}=0$ if and only if

\begin{eqnarray*}
&&J_{2}\Phi \Theta_1 J_{1} \in [\Theta_1 H^2(\oL)]^* +\widetilde{\Theta}_2H^2(\oL)\\
&\Leftrightarrow&\Phi \Theta_1\in J_2[\Theta_1 H^2(\oL)]^*J_1 +J_2\widetilde{\Theta}_2 H^2(\oL)J_1\cr
&\Leftrightarrow&\Phi \in J_2 [\Theta_1 H^2(\oL)]^*J_1\Theta_1^{*} +J_2\widetilde{\Theta}_2 H^2(\oL)J_1\Theta_1^{*}\cr
&\Leftrightarrow&\Phi \in J_2 [ H^2(\oL)]^*\Theta_1^{*}\Theta_1 J_1 +J_2 \widetilde{\Theta}_2H^2(\oL)\Theta_1J_1\cr
&\Leftrightarrow&\Phi \in J_2 [ H^2(\oL)]^* J_1 +J_2 \widetilde{\Theta}_2H^2(\oL)\Theta_1J_1.
\end{eqnarray*}
Hence we complete the proof.
\end{proof}

\section{Characterization of MATHO's}
We will use the results of matrix matrix valued asymmetric truncated Toeplitz operators to characterize matrix valued asymmetric Hankel operators. Recall that for the inner function $\Theta$ we have $S_{\Theta}=C_{\Theta}S_{\Theta}^{*}C_{\Theta}$ and $\J^*S_{\widetilde{\Theta}}\J^*=S_{\Theta}$. Furthermore, the defect operators are given by $D_{\Theta}=I-S_{\Theta}S_{\Theta}^*$ and $\widetilde{D}_{\Theta}= I - S_{\Theta}^*S_{\Theta}$. \par
A simple computation shows that the following relation holds
\begin{equation}
D_{\Theta}C_{\Theta}=C_{\Theta}\widetilde{D}_{\Theta} \quad \text{and} \quad \J^{*}D_{\widetilde{\Theta}}=D_{\Theta}\J^{*}.
\end{equation}
The following theorems are generalization of \cite[Proposition 4.2]{Gu}.
\smallskip
\begin{theorem}\label{characterization 1}
Let $\Theta_1, \Theta_2$ be two pure inner functions and let $B:K_{\Theta_1}\to K_{\Theta_2}$ be a bounded linear operator. Then $B\in \mathcal{MH}(\Theta_1, \Theta_2)$ if and only if there exist $B_{1}^{'}:\widetilde{\mathcal{D}}_{\Theta_1}\to K_{\Theta_2}$ and $B_{2}^{'}:\mathcal{D}_{\Theta_2}\to K_{\Theta_1}$ such that
\begin{equation}
B-S_{\Theta_2}BS_{\Theta_1}=B_{1}^{'}\widetilde{D}_{\Theta_1}+D_{\Theta_2}B^{'*}_2.
\end{equation}
\end{theorem}
\begin{proof}
Let $B=B_{\Phi}^{\Theta_1, \Theta_2}\in \mathcal{MH}(\Theta_1, \Theta_2)$. By Proposition \ref{Prop: a,b,c,d,e,f}(f) we have
$$\J_2^{*}BC_{\Theta_1}=A_{J_2\Phi \Theta J_1}^{\Theta_1, \widetilde{\Theta}_2}\in \mathcal{MT}(\Theta_1, \widetilde{\Theta}_2).$$ Therefore by Theorem \ref{tcara} we have
\begin{align*}
&\J_{2}^{*}BC_{\Theta_1}-S_{\Theta_2}\J_{2}^{*}BC_{\Theta_1}S_{\Theta_1}^* =B_1\dfc{1}+\dfct{2}B_2^*\\
\Rightarrow &~ B-\J_{2}^{*}S_{\Theta_2}\J_{2}^{*}BC_{\Theta_1}S_{\Theta_1}^*C_{\Theta_1}=\J_{2}^{*}B_1\dfc{1}C_{\Theta_1}+\J_2^{*}\dfct{2}B_2^*C_{\Theta_1}\\
\Rightarrow &~ B-S_{\Theta_2}BS_{\Theta_1}=\J_{2}^{*}B_1C_{\Theta_1}\dcf{1}+\dfc{2}\J_2^{*}B_2^*C_{\Theta_1}\\
\Rightarrow &~ B-S_{\Theta_2}BS_{\Theta_1}=B_1^{'}\dcf{1}+\dfc{2}B_2^{'*}
\end{align*}
where $B_1^{'}=\J_{2}^{*}B_1C_{\Theta_1}$ and $B_2^{'}=(C_{\Theta_1}B_2\J_2^{*})^*$.
\end{proof}
\smallskip
\begin{corollary}\label{characterization ii}
Let $\Theta_1, \Theta_2$ be two pure inner functions and let $B:K_{\Theta_1}\to K_{\Theta_2}$ be a bounded linear operator. Then $B\in \mathcal{MH}(\Theta_1, \Theta_2)$ if and only if there exist $\widetilde{B}_{1}:\widetilde{\mathcal{D}}_{\Theta_1}\to K_{\Theta_2}$ and $\widetilde{B}_{2}:\widetilde{\mathcal{D}}_{\Theta_2}\to K_{\Theta_1}$ such that
\begin{equation}
S_{\Theta_2}^*B-BS_{\Theta_1}=\widetilde{B}_{1}\dcf{1}+\dcf{2}\widetilde{B}^{*}_2.
\end{equation}
\end{corollary}
\begin{proof}
Let $B=B_{\Phi}^{\Theta_1, \Theta_2}\in \mathcal{MH}(\Theta_1, \Theta_2)$. By Theorem \ref{characterization 1} and using the defect operators $I-\dcf{2}=S_{\Theta_2}^{*}S_{\Theta_2}$ we have
\begin{align*}
& B-S_{\Theta_2}BS_{\Theta_1}=B_1^{'}\dcf{1}+\dfc{2}B_2^{'*}\\
\Rightarrow &~ S_{\Theta_2}^{*}B-S_{\Theta_2}^{*}S_{\Theta_2}BS_{\Theta_1}^*=S_{\Theta_2}^{*}B_1^{'}\dcf{1}+S_{\Theta_2}^{*}\dfc{2}B_2^{'*}\\
\Rightarrow &~ S_{\Theta_2}^{*}B-(I-\dcf{2})BS_{\Theta_1}=S_{\Theta_2}^{*}B_1^{'}\dcf{1}+\dcf{2}S_{\Theta_2}^{*}B_2^{'*}\\
\Rightarrow &~S_{\Theta_2}^{*}B-BS_{\Theta_1}=S_{\Theta_2}^{*}B_1^{'}\dcf{1}+\dcf{2}S_{\Theta_2}^{*}B_2^{'*}-\dcf{2}BS_{\Theta_1}\\
\Rightarrow &~S_{\Theta_2}^{*}B-BS_{\Theta_1}=S_{\Theta_2}^{*}B_1^{'}\dcf{1}+\dcf{2}(S_{\Theta_2}^{*}B_2^{'*}-BS_{\Theta_1})\\
\Rightarrow &~S_{\Theta_2}^*B-BS_{\Theta_1}=\widetilde{B}_{1}\dcf{1}+\dcf{2}\widetilde{B}^{*}_2
\end{align*}
where $\widetilde{B}_{1}=S_{\Theta_2}^{*}B_1^{'}: \widetilde{\mathcal{D}}_{\Theta_1}\to K_{\Theta_2}$ and $\widetilde{B}^{*}_2=(P_{\widetilde{\mathcal{D}}_{\Theta_2}}(S_{\Theta_2}^{*}B_2^{'*}-BS_{\Theta_1}))^{*}:\widetilde{\mathcal{D}}_{\Theta_2}\to K_{\Theta_1}$.
\end{proof}

\smallskip
\begin{theorem}\label{characterization iii}
Let $\Theta_1, \Theta_2$ be two pure inner functions and let $B:K_{\Theta_1}\to K_{\Theta_2}$ be a bounded linear operator. Then $B\in \mathcal{MH}(\Theta_1, \Theta_2)$ if and only if there exist $\widehat{B}_{1}:\mathcal{D}_{\Theta_1}\to K_{\Theta_2}$ and $\widehat{B}_{2}:\widetilde{\mathcal{D}}_{\Theta_2}\to K_{\Theta_1}$ such that
\begin{equation}
B-S_{\Theta_2}^*BS_{\Theta_1}^*=\widehat{B}_1\dfc{1}+\dcf{2}\widehat{B}^{*}_2.
\end{equation}
\end{theorem}

\begin{proof}
Let $A=A_{\Phi}^{\Theta_1, \Theta_2}\in \mathcal{MT}(\Theta_1, \Theta_2)$. By Proposition \ref{Prop: a,b,c,d,e,f}(e), we have
$$C_{\Theta_2}A_{\Phi}^{\Theta_1, \Theta_2}\J_{1}^*=B_{J_2\Theta(\bar z)^* \Phi(\bar z) J_1}^{\widetilde{\Theta}_1,\Theta_2}=B_{\Psi}^{\widetilde{\Theta}_1,\Theta_2}$$ and so
$$A_{\Phi}^{\widetilde{\Theta}_1, \Theta_2}=C_{\Theta_2}B_{\Psi}^{\Theta_1,\Theta_2}\J_{1}^*=C_{\Theta_2}B\J_{1}^*\in \mathcal{MT}(\widetilde{\Theta}_1, \Theta_2)$$ where $B=B_{\Psi}^{\Theta_1,\Theta_2}.$
Therefore by Theorem \ref{tcara}, we have
\begin{align*}
& C_{\Theta_2}B\J_{1}^{*}-S_{\Theta_2}C_{\Theta_2}B\J_{1}^{*}S_{\widetilde{\Theta}_1}^* =B_1\dfct{1}+\dfc{2}B_2^*\\
\Rightarrow & B-C_{\Theta_2}S_{\Theta_2}C_{\Theta_2}B\J_{1}^{*}S_{\widetilde{\Theta}_1}^*\J_{1}^{*} =C_{\Theta_2}B_1\dfct{1}\J_{1}^{*}+C_{\Theta_2}\dfc{2}B_2^*\J_{1}^{*}\\
\Rightarrow& B-S_{\Theta_2}^*BS_{\Theta_1}^*=C_{\Theta_2}B_1\J_{1}^{*}\dfc{1}+\dcf{2}C_{\Theta_2}B_2^*\J_{1}^*\\
\Rightarrow& B-S_{\Theta_2}^*BS_{\Theta_1}^*=\widehat{B}_1\dfc{1}+\dcf{2}\widehat{B}_{2}^{*}
\end{align*}
where $\widehat{B}_1=C_{\Theta_2}B_1\J_{1}^{*}: \mathcal{D}_{\Theta_1}\to K_{\Theta_2}$ and $\widehat{B}_2=(C_{\Theta_2}B_2^*\J_{1}^*)^*:\widetilde{\mathcal{D}}_{\Theta_2}\to K_{\Theta_1}$.
\end{proof}

\begin{corollary}\label{characterization iv}
Let $\Theta_1, \Theta_2$ be two pure inner functions and let $B:K_{\Theta_1}\to K_{\Theta_2}$ be a bounded linear operator. Then $B\in \mathcal{MH}(\Theta_1, \Theta_2)$ if and only if there exist $\widehat{B}_{1}^{'}:\mathcal{D}_{\Theta_1}\to K_{\Theta_2}$ and $\widehat{B}_{2}^{'}:\mathcal{D}_{\Theta_2}\to K_{\Theta_1}$ such that
\begin{equation}
S_{\Theta_2}B-BS_{\Theta_1}^*=\widehat{B}_{1}^{'}\dfc{1}+\dfc{2}\widehat{B}^{'*}_2.
\end{equation}
\end{corollary}

\begin{proof}
By Theorem \ref{characterization iii},
 $B\in \mathcal{MH}(\Theta_1, \Theta_2)$ if and only if
$$B-S_{\Theta_2}^*BS_{\Theta_1}^*=\widehat{B}_1\dfc{1}+\dcf{2}\widehat{B}^{*}_2,$$
which is equivalent to
$$S_{\Theta_2}B-S_{\Theta_2}S_{\Theta_2}^*BS_{\Theta_1}^*=S_{\Theta_2}\widehat{B}_1\dfc{1}+S_{\Theta_2}\dcf{2}\widehat{B}^{*}_2.$$

By the defect operator $\dfc{2}=I-S_{\Theta_2}S_{\Theta_2}^*$, we have
 \begin{align*}
&S_{\Theta_2}B-(I-\dfc{2})BS_{\Theta_1}^*=S_{\Theta_2}\widehat{B}_1\dfc{1}+S_{\Theta_2}\dcf{2}\widehat{B}^{*}_2\\
\Rightarrow& S_{\Theta_2}B-BS_{\Theta_1}^*=S_{\Theta_2}\widehat{B}_1\dfc{1}+S_{\Theta_2}\dcf{2}\widehat{B}^{*}_2-\dfc{2}BS_{\Theta_1}^*\\
\Rightarrow& S_{\Theta_2}B-BS_{\Theta_1}^*=S_{\Theta_2}\widehat{B}_1\dfc{1}+\dfc{2}S_{\Theta_2}\widehat{B}^{*}_2-\dfc{2}BS_{\Theta_1}^*\\
\Rightarrow& S_{\Theta_2}B-BS_{\Theta_1}^*=S_{\Theta_2}\widehat{B}_1\dfc{1}+\dfc{2}(S_{\Theta_2}\widehat{B}^{*}_2-BS_{\Theta_1}^*)\\
\Rightarrow& S_{\Theta_2}B-BS_{\Theta_1}^*=\widehat{B}_{1}^{'}\dfc{1}+\dfc{2}\widehat{B}^{'*}_2
\end{align*}
where $\widehat{B}_{1}^{'}=S_{\Theta_2}\widehat{B}_1: \mathcal{D}_{\Theta_1}\to K_{\Theta_2}$ and $\widehat{B}_{2}^{'}=(P_{\mathcal{D}_2}(S_{\Theta_2}\widehat{B}^{*}_2-BS_{\Theta_1}^*))^*:\mathcal{D}_{\Theta_2}\to K_{\Theta_1}$.
\end{proof}

\section{Characterization by modified compressed shift}
In the scalar case the compressed shift was introduced in \cite[section 10]{sar}.
The notion of modified compressed shift was introduced in section \ref{char by of MTTO by mod comp shift} (see \cite{KT}). Characterization of matrix valued asymmetric truncated Toeplitz operators by using modified compressed shift was presented in \cite{RK3}. In this part we use the modified compressed shift to characterize MATHO's.
\begin{theorem}
Let $\Theta_1, \Theta_2$ be two pure inner functions and let $B:K_{\Theta_1}\to K_{\Theta_2}$ be a bounded linear operator. Then $B\in \mathcal{MH}(\Theta_1, \Theta_2)$ if and only if one (and all) of the following conditions holds.
\item[(a)] $B-S_{\Theta_2, X_2}BS_{\Theta_1, X_1}=\mathbb{B}_1\dfc{1}+\dfc{2}\mathbb{B}_{2}^{'*}$.
\item[(b)] $S_{\Theta_2, X_2}^*B-BS_{\Theta_1, X_1}=\widetilde{\mathbb{B}}_1\dcf{1}+\dcf{2}\widetilde{\mathbb{B}}_{2}^{'*}$.
\item[(c)] $B-S_{\Theta_2, X_2}^*BS_{\Theta_1, X_1}^*=\mathbb{B}_3\dcf{1}+\dcf{2}\mathbb{B}_{4}^{'*}$.
\item[(d)] $S_{\Theta_2, X_2}B-BS_{\Theta_1, X_1}^*=\widetilde{\mathbb{B}}_3\dfc{1}+\dfc{2}\widetilde{\mathbb{B}}_{4}^{'*}$.
\end{theorem}
\begin{proof}
The proof is analogous to the proof of Theorem \ref{characterization 1}.
\end{proof}
\smallskip
\begin{corollary}\label{S. Invar}
Let $B: K_{\Theta_1}\to K_{\Theta_2}$ be a bounded linear operator. Then $B\in \mathcal{MH}(\Theta_1, \Theta_2)$ if and only if it has one (and all) of the following properties
\item[(a)] $\langle BSf, S^*g\rangle=\langle Bf, g\rangle~\text{for all} ~ f\in K_{\Theta_1},~ g\in K_{\Theta_2}~ \text{such that } ~f\in \widetilde{\mathcal{D}}_{\Theta_1}^{\perp} ~\text{and}~ g\in \mathcal{D}_{\Theta_2}^{\perp}$.\label{S. Invar a}
\item[(b)] $\langle Bf, Sg\rangle=\langle BSf, g\rangle~\text{for all} ~ f\in K_{\Theta_1},~ g\in K_{\Theta_2}~ \text{such that } ~f\in \widetilde{\mathcal{D}}_{\Theta_1}^{\perp} ~\text{and}~ g\in \widetilde{\mathcal{D}}_{\Theta_2}^{\perp}$.\label{S. Invar b}
\item[(c)] $\langle BS^*f, Sg\rangle=\langle Bf, g\rangle~\text{for all} ~ f\in K_{\Theta_1},~ g\in K_{\Theta_2}~ \text{such that } ~f\in \mathcal{D}_{\Theta_1}^{\perp} ~\text{and}~ g\in \widetilde{\mathcal{D}}_{\Theta_2}^{\perp}$.\label{S. Invar c}
\item[(d)] $\langle Bf, S^*g\rangle=\langle BS^*f, g\rangle~\text{for all} ~ f\in K_{\Theta_1},~ g\in K_{\Theta_2}~ \text{such that } ~f\in \mathcal{D}_{\Theta_1}^{\perp} ~\text{and}~ g\in \mathcal{D}_{\Theta_2}^{\perp}$.\label{S. Invar d}
\end{corollary}

\begin{proof}
 (a) Let $f\in \widetilde{\mathcal{D}}_{\Theta_1}^{\perp}$ and $g\in \mathcal{D}_{\Theta_2}^{\perp}$. Then
 $Sf=S_{\Theta_1}f$ and $S^*g=S_{\Theta_2}g$ and hence
 \begin{align*}
 \langle BSf, S^*g\rangle
 &=\langle BS_{\Theta_1}f, S_{\Theta_2}^*g\rangle=\langle S_{\Theta_2}BS_{\Theta_1}f, g\rangle.
 \end{align*}
 By Theorem \ref{characterization 1}, we have
 \begin{align*}
 \langle BSf, S^*g\rangle
 &=\langle (B-(B_1^{'}\dcf{1}+\dfc{2}B_{2}^{'*}))f, g\rangle\\
 &=\langle Bf, g\rangle-\langle B_1^{'}\dcf{1}f,g\rangle-\langle\dfc{2}B_{2}^{'*}f, g\rangle
 \end{align*}
 where $B_{1}^{'}: \widetilde{D}_{\Theta_1}\to K_{\Theta_2}$ and $B_2^{'}:\mathcal{D}_{\Theta_2}\to K_{\Theta_1}$.
 Since $f\in \widetilde{\mathcal{D}}_{\Theta_1}^{\perp}$. This implies that $S_{\Theta_1}f=zf(z)$ and hence
 $S_{\Theta_1}^*S_{\Theta_1}f=f$ means that
 $\dcf{1}f=(I-S_{\Theta_1}^*S_{\Theta_1})f=0$. Therefore
 $$\langle B_1^{'}\dcf{1}f, g\rangle=0.$$
 Furthermore, since $f\in\widetilde{\mathcal{D}}_{\Theta_1}^{\perp}$ implies that $\dfc{2}B_2^{'*}f\in \mathcal{D}_{\Theta_2}$ and $f\in \mathcal{D}_{\Theta_2}^{\perp}$ follows that
 $\langle \dfc{2}B_2^{'*}f, g\rangle=0$. Therefore, we have
 $$ \langle BSf, S^*g\rangle=\langle Bf, g\rangle.$$
 The proofs of (b), (c) and (d) are analogous and is left to the reader.
\end{proof}

\end{document}